\documentclass[11pt]{amsart}
\usepackage{amsthm}
\usepackage{amsmath}
\usepackage[textsize=tiny, disable]{todonotes}
\usepackage[margin=1.2in]{geometry}
\usepackage{amssymb,verbatim,enumitem}
\usepackage{tikz}
\usepackage{tikz-cd}
\usepackage{multirow}

\theoremstyle{plain}
\newtheorem{theorem}{Theorem}[section]          
\newtheorem{corollary}[theorem]{Corollary}     
\newtheorem{lemma}[theorem]{Lemma}              
\newtheorem{proposition}[theorem]{Proposition}

\newtheorem{maintheorem}{Theorem}

\theoremstyle{definition}

\newtheorem{remark}[theorem]{Remark}

\numberwithin{table}{section}

\newcommand{\Z}{\mathbb{Z}}\newcommand{\Q}{\mathbb{Q}}
\newcommand{\R}{\mathbb{R}}\newcommand{\C}{\mathbb{C}}

\newcommand{\QQ}{\mathbb{Q}}

\newcommand{\tensor}{\otimes}
\newcommand{\embedded}{\hookrightarrow}

\newcommand{\V}{\mathcal{V}}
\newcommand{\K}{\mathcal{K}}
\newcommand{\g}{\mathfrak{g}}

\title{Rational ellipticity of $G$-manifolds from their quotients}
\author{Elahe Khalili Samani}
\address[E. Khalili Samani]{Department of Mathematics, \newline \indent Clark University,\newline \indent  Worcester, MA 01610, USA}
\email{ekhalilisamani@clarku.edu}
\author{Marco Radeschi}
\address[M. Radeschi]{Universit\`a degli Studi di Torino,\newline \indent Dipartimento di matematica ``G. Peano'',\newline \indent Via Carlo Alberto 10,\newline \indent 10123 Torino (TO), Italy}
\email{marco.radeschi@unito.it}


\begin{document}

\begin{abstract} We prove that if a compact, simply connected Riemannian $G$-manifold $M$ has orbit space $M/G$ isometric to some other quotient $N/H$ with $N$ having zero topological entropy, then $M$ is rationally elliptic. This result, which generalizes most conditions on rational ellipticity, is a particular case of a more general result involving manifold submetries.
\end{abstract}

\maketitle

\section{Introduction}

When studying Riemannian manifolds with lower sectional curvature bound, one important principle guiding the open problems in the area is that non-negative sectional curvature should somehow imply ``simple topology''. One of the most precise and far-reaching versions of this principle is the Bott-Grove-Halperin conjecture, stating that a closed simply-connected Riemannian manifold with non-negative curvature is \emph{rationally elliptic}, that is, the total rational homotopy $\pi_*(M)\otimes \Q:=\bigoplus_{i\geq 2} \pi_i(M)\otimes \Q$ is finite dimensional. Since by \cite{FH79} rational ellipticity is known to imply a number of important topological restrictions conjectured to hold in non-negative curvature, such as a sharp upper bound on the total Betti numbers or the non-negativity of the Euler characteristic (cf. for example \cite{FHT01}, Ch. 32 and references therein), finding geometric conditions that imply rational ellipticity is particularly desirable.

On the one hand, in \cite{PP06} Paternain and Petean proved that rational ellipticity follows from a dynamical concept called \emph{zero topological entropy} (cf. Section \ref{SS:topent} for the definition). On the other hand, a series of papers
shows that a closed, simply connected Riemannian manifold $M$ with an isometric action by a compact Lie group $G$ is rationally elliptic, under either of the following conditions:
\begin{itemize}
\item $\dim M/G=1$ (Grove-Halperin \cite{GH87}).
\item The $G$-action is \emph{polar} (cf. Section \ref{S:polar}) and the section is either flat or spherical (Grove-Ziller \cite{GZ12}).
\item $M$ is \emph{almost non-negatively curved} and $\dim M/G=2$ (Grove-Wilking-Yeager \cite{GWY19}).
\end{itemize}

The main result of this paper is a new geometric condition to rational ellipticity, which somehow merges the symmetry and topological entropy conditions into a unique framework:

\begin{maintheorem}\label{main-cor:G-manifolds}
Let $M$ and $N$ be closed Riemannian manifolds which admit isometric actions by compact Lie groups $G$ and $H$, respectively. Assume that the quotient spaces $M/G$ and $N/H$ are isometric. If $M$ is simply connected and $N$ has zero topological entropy, then $M$ is rationally elliptic. 
\end{maintheorem}

\begin{remark}
Theorem \ref{main-cor:G-manifolds} implies all the results mentioned above, since:
\begin{enumerate}
\item Applying the theorem with $M=N$ and $G=H=\{e\}$ implies the result of Paternain and Petean \cite{PP06}.
It should be mentioned, however, that the result of Paternain and Petean also applies when $M$ is not simply connected.
\item If $\dim M/G=1$, then applying the theorem with $N=S^1$, $H=\Z_2$ (acting on $S^1$ as a reflection) gives the result of Grove-Halperin \cite{GH87}.
\item If the $G$-action is \emph{polar} and the section is flat or spherical, then letting $N$ be a compact quotient of the section, and $H$ the Weyl group acting on $N$, one gets the result of Grove-Ziller \cite{GZ12}.
\item If $M$ is \emph{almost non-negatively curved} and $\dim M/G=2$, then it was proved in \cite{GWY19} that up to changing the metric on $M$, the quotient $M/G$ is isometric to either $S^3/S^1$, $S^2/\Gamma$ or $T^2/\Gamma'$ for some discrete groups of isometries $\Gamma, \Gamma'$. In either case, there is an obvious choice of $N$ and $H$ in Theorem \ref{main-cor:G-manifolds} which implies the result of Grove-Yeager-Wilking \cite{GWY19}.
\end{enumerate}
At the same time, however, our result uses in a fundamental way the results in the aforementioned papers, especially the approaches of Paternain-Petean and Grove-Yeager-Wilking.
\end{remark}

\begin{remark}
By the result of \cite{PP06}, the manifold $N$ in Theorem \ref{main-cor:G-manifolds} is in particular rationally elliptic. We do not know whether it is possible to relax the assumption of Theorem \ref{main-cor:G-manifolds} to $N$ being rationally elliptic. If this was possible, it would mean that the rational ellipticity of a $G$-manifold $M$ depends only on the geometry of the orbit space $M/G$.
\end{remark}

Theorem \ref{main-cor:G-manifolds} is a special case of the following more general result, in which quotients by isometric actions get replaced by \emph{manifold submetries} (cf. Section \ref{SS:quotient geodesics}):
\begin{maintheorem}\label{main-thm: foliations}
\todo{Changed phrasing of theorem (Point 1)}
Let $M, N$ be closed Riemannian manifolds, with $M$ simply connected and $N$ with zero topological entropy. Suppose $\pi_M:M\to X$ and $\pi_N:N\to X$ are manifold submetries. If the generic fibers of $\pi_M$ are connected with finite fundamental group and rationally elliptic universal cover, then $M$ is rationally elliptic. 
\end{maintheorem}

Recall that an action of a Lie group on a Riemannian manifold is called
\emph{infinitesimally polar} if all of its isotropy representations are polar.
By Lytchak-Thorbergsson \cite{LT10}, this is equivalent to the quotient space of the action being an orbifold.
As an application of Theorem \ref{main-cor:G-manifolds}, we prove the following:

\begin{maintheorem}\label{main-thm:infinitesimally polar}
Suppose $M$ is a closed, simply connected, non-negatively curved Riemannian manifold. If $M$ admits an isometric,
cohomogeneity three, infinitesimally polar action by a compact Lie group $G$, then $M$ is rationally elliptic. 
\end{maintheorem}

We believe that Theorem \ref{main-thm:infinitesimally polar} could also be proved using the techniques in Grove-Wilking-Yeager \cite{GWY19}, although that would involve more of a case-by-case study of the possible quotient spaces. On the other hand, we believe Theorem \ref{main-cor:G-manifolds} could be used in other situations as well, such as, proving rational ellipticity for non-negatively curved manifold of cohomogeneity 3 (without the infinitesimally polar condition), at least for positively curved manifolds. This is the subject of an ongoing project, which will be the topic of a forthcoming paper.

The structure of the paper is as follows: In Section \ref{S:preliminaries}, we recall the basic facts about the concepts used throughout the paper. In Section \ref{S:G-manifolds}, we prove Theorem \ref{main-thm: foliations} and Theorem \ref{main-cor:G-manifolds} by generalizing a fundamental technical lemma of \cite{PP06}. Finally, in Section \ref{S:infinitesimally polar}, we prove Theorem \ref{main-thm:infinitesimally polar}.

\subsection*{Acknowledgements}
The authors wish to thank Fred Wilhelm for his interest in the manuscript, and helpful feedback. Furthermore, the authors wish to thank the anonymous referee for thoroughly reading the manuscript, spotting imprecisions, and suggesting improvements and corrections.

\section{Preliminaries}\label{S:preliminaries}

\subsection{Alexandrov spaces}\label{SS:Alexandrov}

A metric space $(X,d)$ is called a \emph{length space} if for any pair of points $x, y$ in $X$
the distance between $x$ and $y$ is the infimum of the lengths of curves connecting $x$ and $y$, and a \emph{geodesic space} if this infimum is always achieved by some shortest curve. An \emph{Alexandrov space} is a geodesic space 
with a lower curvature bound in the comparison geometry sense (for more details, see \cite{BBI01} and \cite{BGP92}).
Given an Alexandrov space $X$, a curve $\gamma:I\to X$ which locally minimizes the distance is called a \emph{geodesic}.
If $\gamma_1$ and $\gamma_2$ are two geodesics starting at $x\in X$, the angle between them is defind as the limit
$$\lim_{y, z\to x}\measuredangle yxz,$$
where $y$ and $z$ are points on $\gamma_1$ and $\gamma_2$, respectively. Two geodesics starting at $x$ 
are called equivalent if the angle between them equals zero. Let $\Sigma'_xX$ denote the set of equivalence classes 
of geodesics starting at $x$. Then $\Sigma'_xX$ forms a metric space in which the distance is the angle between the geodesics.
The metric completion of $\Sigma'_xX$, denoted by $\Sigma_xX$, is called the \emph{space of directions} of $X$ at $x$.
The cone over $\Sigma_xX$, which is defined by
$$T_xX:=\left(\Sigma_xX\times[0,\infty)\right)/{(x_1,0)\sim (x_2,0)},$$
is called the \emph{tangent cone} of $X$ at $x$. 

\subsection{Manifold submetries}\label{SS:quotient geodesics}

Let $M$ be a Riemannian manifold, and let $X$ be a metric space. A manifold submetry from $M$ to $X$
is a continuous map $\pi:M\to X$ such that the fibers of $\pi$ are submanifolds of $M$, and moreover,
$\pi$ sends metric balls to metric balls of the same radius. Given a compact Riemannian manifold $M$
and a manifold submetry $\pi:M\to X$, the fibers of $\pi$ are equidistant. If $M$ has a lower bound $\kappa$ on its sectional curvature, then $X$ is an Alexandrov space with the same lower curvature bound $\kappa$.

\begin{remark}
It is known that, given a manifold submetry $\pi:M\to X$ and a vector $x\in TM$ perpendicular to the $\pi$-fiber through that point, the geodesic $\exp(tx)$ is perpendicular to all the $\pi$-fibers it meets (cf. for example \cite{LK22}, Proposition 12.5). Such a geodesic is called a \emph{horizontal geodesic}.

One important observation made in \cite[Lemma 12]{MR20}, which will be useful to us, is that given a manifold submetry $\pi:M\to X$, the projections on $X$ of horizontal geodesics in $M$ can be defined metrically purely in terms of $X$ - such projections are called \emph{quotient geodesics}. As a consequence, given manifold submetries $\pi:M\to X$ and $\pi':M'\to X$ over the same base $X$, the $\pi$-projections of $\pi$-horizontal geodesics in $M$ coincide with the $\pi'$-projections of the $\pi'$-horizontal geodesics in $M'$, which coincide with the quotient geodesics of $X$.
\end{remark}

\subsection{The topological entropy}\label{SS:topent}

Let $(X,d)$ be a compact metric space, and let $\phi_t:X\to X$ be a continuous flow, that is, a continuous map 
satisfying $\phi_{t+s}=\phi_t\circ\phi_s$. For any $T>0$, define a new metric $d_T$ on $X$ by
$$d_T(x,y)=\max_{0\leq t\leq T}d(\phi_t(x),\phi_t(y)).$$
Let $N_T^{\epsilon}$ denote the minimum number of $\epsilon$-balls in the metric $d_T$ that are needed to cover $X$.
The topological entropy of the flow $\phi_t$ is defined by
$$h_{\text{top}}(\phi_t)=\lim_{\epsilon\to 0}\limsup\limits_{T\to\infty}\frac{1}{T}\log(N_T^{\epsilon}).$$

For a closed Riemannian manifold $M$ with tangent bundle $TM$, the geodesic flow $\phi_t:TM\to TM$
of $M$ is defined by $\phi_t(x,v)=(\gamma_v(t),\gamma'_v(t))$, where $\gamma_v(t)$ is the unique geodesic 
satisfying $\gamma_v(0)=x$ and $\gamma'_v(0)=v$. The geodesic flow leaves invariant the unit tangent bundle $SM$ of $M$ 
and it thus induces a geodesic flow $\phi_t:SM\to SM$. The topological entropy of $M$, denoted by $h_{\text{top}}(M)$, 
is by definition the topological entropy of the geodesic flow $\phi_t:SM\to SM$.

\todo{Section on simplicial approximation removed (Point (3))}

\subsection{Polar actions}\label{S:polar}

An isometric action of a Lie group $G$ on a closed Riemannian manifold $M$ is called polar if there exists a closed, connected, 
immeresed submanifold $\Sigma$ of $M$, called a section, that meets all the $G$-orbits orthogonally.
Given a polar $G$-manifold $M$ with a section $\Sigma$, the normalizer $N(\Sigma)$ and the centralizer $Z(\Sigma)$ 
of $\Sigma$ are the subgroups of $G$ defined by $N(\Sigma)=\{g\in G\mid g\cdot\Sigma=\Sigma\}$
and $Z(\Sigma)=\{g\in G\mid g\cdot p=p, \forall p\in\Sigma\}$. Note that $N(\Sigma)$ acts on $\Sigma$ 
and the kernel of this action equals $Z(\Sigma)$. Hence there is an effective action of ${N(\Sigma)}/{Z(\Sigma)}$ 
on $\Sigma$. The quotient group $W(\Sigma):={N(\Sigma)}/{Z(\Sigma)}$ is called the generalized Weyl group of the action.
It is a well-known fact that the inclusion $\iota:\Sigma\embedded M$ induces an isometry ${\Sigma}/{W(\Sigma)}\to M/G$ 
(see, for example, \cite[Proposition 1.3.2]{Gor22}).

\section{The proof of Theorem \ref{main-thm: foliations}}\label{S:G-manifolds}
\todo{Section title changed. New Initial paragraph, Lemma 3.1, Corollary 3.2, Proposition 3.3 (Point (3))}
The goal of this section is to prove Theorem \ref{main-thm: foliations}. We begin by analyzing the action of $\pi_1(L)$ on $H_*(\tilde{L},\QQ)$ when $L$ is the connected fiber of a manifold submetry $\pi:M\to X$, with $\pi_1(M)=0$ and $\pi_1(L)$ finite.

\begin{lemma}\label{L:nilptriv}
Assume $L$ is a nilpotent manifold with $\pi_1(L)$ finite. Then the action of $\pi_1(L)$ on $H_*(\tilde{L},\QQ)$ is trivial.
\end{lemma}
\begin{proof}
For each $j$ consider the representation $\rho_j:\pi_1(L)\to GL(H_j(\tilde{L},\QQ))$. Since the action is nilpotent, every $g\in \operatorname{Im}(\rho_j)$ is a unipotent matrix. On the other hand, since $\pi_1(L)$ is finite, every $g\in \operatorname{Im}(\rho_j)$ is diagonalizable over $\mathbb{C}$, but the only unipotent diagonalizable matrix is the identity.
\end{proof}

\begin{corollary}
Given a manifold submetry $\pi:M\to X$ with $M$ simply connected, if the generic fiber $L$ is connected and has finite fundamental group, then $\pi_1(L)$ acts trivially on $H_*(\tilde{L},\QQ)$.
\end{corollary}
\begin{proof}
By Lemma \ref{L:nilptriv} it is enough to show that $L$ is a nilpotent space. Recall that $M$ admits a stratification, where each stratum $\Sigma_r$ consists of the union of $r$-dimensional fibers of $\pi$. The result is then a consequence of the following facts:
\begin{itemize}
\item The union $M^{(2)}$ of strata of codimension $\leq 2$ is an open, dense, simply connected submanifold of $M$, containing all the principal fibers. (Appendix B of \cite{MR20}).
\item The fibers of $\pi|_{M^{(2)}}:M^{(2)}\to X$ are leaves of a singular Riemannian foliation $(M^{(2)},\mathcal{F})$. (Lemma 49 of \cite{MR20}).
\item Given a singular Riemannian foliation $(M^{(2)},\mathcal{F})$ on a simply connected manifold with compact leaves and finitely many strata, the principal leaves are nilpotent spaces. (Theorem A of \cite{KR24}).
\end{itemize}
\end{proof}

Given a compact Riemannian manifold $M$, a point $p\in M$, and a closed submanifold $L$ of $M$, we denote by $\Omega(L,p)$ the space of piecewise smooth paths starting from some point in $L$ and ending at $p$. Moreover, for any real number $r$, $\Omega^r(L,p)$ denotes the subset of $\Omega(L,p)$ consisting of paths of length at most $r$.

\todo{Old lemmas here removed (Point (3))}

\begin{proposition}\label{P:Gromov}
Let $M$ be a compact, simply connected Riemannian manifold, and let $L\subseteq M$ be a connected, compact submanifold such that $|\pi_1(L)|<\infty$ and $\pi_1(L)$ acts trivially on $H_*(\tilde{L},\QQ)$. Then there is a constant $C$, depending on $M$ and $L$, such that for every $k\in \mathbb{N}$ the inclusion $\iota:\Omega^{Ck}(L, p)\to \Omega(L,p)$ induces surjective maps $\iota_*:H_{j}(\Omega^{Ck}(L, p),\QQ)\to H_{j}(\Omega(L, p),\QQ)$, $\forall j\leq k$.
\end{proposition}

\begin{proof}
This proof is an adaptation of Lemma 3.1 of \cite{PP95}, where the result was proved in the case $\pi_1(L)=0$.  Introduce simplicial structures on $L$ and its universal cover $\tilde{L}$ such that the covering map $c:\tilde{L}\to L$ is a simplicial map. Up to refinement, we can extend the simplicial structure on $L$ to one in $M$ 
\todo{Added citation (Point (2))}
(cf. \cite{Ver}, Section 8.11) and in particular, we can obtain a simplicial structure on the mapping cylinder
\[
M_c=((\tilde{L}\times[0,1])\coprod M)/((\tilde{p},1)\sim c(p)),\qquad \varphi_c:M_c\to M\quad \left\{\begin{array}{l}\varphi_c(\tilde{p},t)=c(\tilde{p})\\ \varphi_c(q\in M)=q
\end{array}\right.
\]
Recall that the inclusion map $M\to M_c$ and the map $\varphi_c:M_c\to M$ are homotopy equivalences, and $\tilde{L}$ can be identified with the subspace $\tilde{L}\times \{0\}\subseteq M_c$. Finally, we remark that the simplicial structure on $M_c$ can be chosen so that the maps $\tilde{L}\to M_c$, $M\to M_c$ and $M_c\to M$ are simplicial maps.
Notice that the (simplicial) action of $\pi_1(L)$ on $\tilde{L}$ induces an action on $M_c$ by
\[
g\cdot (\tilde{p},t)=(g\cdot \tilde{p},t),\qquad g\cdot (p\in M)=p,
\]
and, in particular, it induces a free action of $\pi_1(L)$ on $\Omega(\tilde{L},p)$, $p\in M$, such that the map $\hat{\varphi}_c:\Omega(\tilde{L},p)\to \Omega(L,p)$, $(\hat{\varphi}_c(\gamma))(t):=\varphi_c(\gamma(t))$ is $\pi_1(L)$-invariant, while the evaluation map $\Omega(\tilde{L},p)\to \tilde{L}$ is $\pi_1(L)$-equivariant. Therefore, in the following diagram with fibration rows, the columns are fibrations as well:
\begin{center}
\begin{tikzcd}
\Omega M_c\arrow[d,"\simeq" ']\arrow{r}&\Omega(\tilde{L},p)\arrow{d}{\hat{\varphi}_c} \arrow{r}& \tilde{L}\arrow{d}{c}\\ 
\Omega M\arrow{r}&\Omega({L},p) \arrow{d} \arrow{r}& L\arrow{d}\\
&B\pi_1(L) \arrow{r}{\operatorname{id}}& B\pi_1(L)
\end{tikzcd}
\end{center}
Since $\pi_1(L)$ acts trivially on $H_*(\tilde{L})$, we can apply the Serre spectral sequence to the fibration $\tilde{L}\to L\to B\pi_1(L)$ and obtain $H_*(\tilde{L},\QQ)\simeq H_*(L,\QQ)$. Recall furthermore that in the fibration $\Omega M\to \Omega(L,p)\to L$, the action of $\pi_1(L)$ on $\Omega M$ factors through $\pi_1(L)\to \pi_1(M)=0$ (cf. Exercise 4.3.10 in \cite{Hat02}), thus again the Serre spectral sequence can be applied. Finally, since the spectral sequence of $\Omega M_c\to \Omega(\tilde{L},p)\to \tilde{L}$ is the pullback of the spectral sequence for $\Omega M\to \Omega(L,p)\to L$ we obtain that $\hat{\varphi}_c:H_*(\Omega(\tilde{L},p),\QQ)\to H_*(\Omega(L,p),\QQ)$ is an isomorphism.

Since $\tilde{L}$ is simply connected, the same arguments of Lemma 2.1 in \cite{PP95} go through, and we obtain the existence of a map $\alpha:(M_c,\tilde{L})\to (M_c, \tilde{L})$ with the following properties:
\begin{enumerate}
\item the map $\hat{\alpha}:\Omega(\tilde{L},p)\to \Omega(\tilde{L},\alpha(p))$ given by $\hat{\alpha}(\gamma)(t)=\alpha(\gamma(t))$ induces the identity in homology.
\item $\alpha$ sends the 1-skeleton of $M_c$ to a point.
\end{enumerate}
Letting $\beta=\varphi_c\circ \alpha:(M_c,\tilde{L})\to (M,L)$ and $\hat{\beta}=\hat{\pi}_c\circ \hat{\alpha}:\Omega(\tilde{L},p)\to \Omega(L, \beta(p))$, which from the previous paragraph induce isomorphisms in rational cohomology, we can then apply the arguments of Lemma 3.1 in \cite{PP95} as follows. Define the subspace $\Omega^{pl}(\tilde{L},p)\subset \Omega(\tilde{L},p)$ of paths that are linear on each simplex of $M_c$, endowed with a cellular structure with cells $\Phi:\Delta_1\times\ldots \Delta _m\to \Omega^{pl}(\tilde{L},p)$, where $\Delta_1,\ldots, \Delta_m$ is a sequence of simplices in $M_c$ such that $\Delta_1\subseteq \tilde{L}$ and $\Delta_i,\Delta_{i+1}$ are faces of a common simplex, and where $\Phi(q_1,\ldots q_m)$ is the unique piecewise linear curve passing through the $q_i$'s. Letting $\Omega^{pl}(\tilde{L},q)^{(k)}$ denote the $k$-skeleton, the same computations as Lemma 3.1 in \cite{PP95} give that for $\gamma\in \Omega^{pl}(\tilde{L},q)^{(k)}$
\[
\operatorname{length}(\hat{\beta}(\gamma))=\operatorname{length}(\hat{\beta}(\gamma))-\operatorname{length}(\hat{\beta}(\gamma|_{M_c^{(1)}}))\leq Ck,
\]
where $C$ depends on the piecewise linear metric in $M_c$ and the Lipschitz constant of $\beta$. From this we have $\hat{\beta}(\Omega^{pl}(\tilde{L},q)^{(k)})\subseteq \Omega^{Ck}(L,q)$, and from the commutative diagram in rational cohomology
\begin{center}
\begin{tikzcd}
H_j(\Omega^{pl}(\tilde{L},q)^{(k)},\QQ) \arrow[d, "\hat{\beta}"]\arrow[r,"\simeq"]&H_j(\Omega^{pl}(\tilde{L},q),\QQ)  \arrow[d,"\hat{\beta}", "\simeq" ']\\
H_j(\Omega^{Ck}({L},q),\QQ) \arrow{r}&H_j(\Omega({L},q),\QQ)
\end{tikzcd}
\end{center}
for all $j\leq k$, we have the result.
\end{proof}
We now proceed to the proof of Theorem \ref{main-thm: foliations}.

\begin{proof}[Proof of Theorem \ref{main-thm: foliations}]
 Let $X^{\text{reg}}$ denote the set of manifold points of $X$, 
and let $N^{\text{reg}}:=\pi_N^{-1}(X^{\text{reg}})$. Fix $x\in X^{\text{reg}}$. 
By Theorem $3.27$ in \cite{Pat99}, we have:
\begin{equation}\label{eq:3.27}
h_{\text{top}}(N)\geq\limsup_{T\to\infty}\frac{1}{T}\log\int_Nn_T(\pi_N^{-1}(x),\bar{y})d\bar{y}=\limsup_{T\to\infty}\frac{1}{T}\log\int_{N^{\text{reg}}}n_T(\pi_N^{-1}(x),\bar{y})d\bar{y},\vspace{0.1cm}
\end{equation}
where $n_T(\pi_N^{-1}(x),\bar{y})$ denotes the number of horizontal geodesics of length at most $T$ between $\pi_N^{-1}(x)$ and $\bar{y}$. Since $\pi_N|_{N^{\text{reg}}}:N^{\text{reg}}\to X^{\text{reg}}$ is a Riemannian submersion, we can apply Fubini's Theorem to Equation \eqref{eq:3.27}:
\begin{align*}
h_{\text{top}}(N) &  \geq\limsup_{T\to\infty}\frac{1}{T}\log\int_{N^{\text{reg}}}n_T(\pi_N^{-1}(x),\bar{y})d\bar{y}\\
 & =\limsup_{T\to\infty}\frac{1}{T}\log\int_{X^{\text{reg}}}\left(\int_{\pi_N^{-1}(y)}n_T(\pi_N^{-1}(x),\bar{y})d\bar{y}|_{\pi_N^{-1}(y)}\right)dy.\vspace{0.1cm} 
\end{align*}

Now, choose $y\in X^{\text{reg}}$ and $\bar{y}\in\pi_N^{-1}(y)$. As discussed in Section \ref{SS:quotient geodesics}, 
every quotient geodesic in $X$ is the image under $\pi_N$ of a horizontal geodesic in $N$. Since $y$ is a manifold point of $X$, every quotient geodesic from $x$ to $y$ is the image of a unique horizontal geodesic between $\pi_N^{-1}(x)$ and $\bar{y}$. Therefore, letting $n_T(x,y)$ denote the number of quotient geodesics between $x$ and $y$, 
we get $n_T(\pi_N^{-1}(x),\bar{y})=n_T(x,y)$ and hence
\begin{align*}
h_{\text{top}}(N) & \geq\limsup_{T\to\infty}\frac{1}{T}\log\int_{X^{\text{reg}}}\left(\int_{\pi_N^{-1}(y)}n_T(x,y)d\bar{y}|_{\pi_N^{-1}(y)}\right)dy\\
 & =\limsup_{T\to\infty}\frac{1}{T}\log\int_{X^{\text{reg}}}n_T(x,y){\text{Vol}}(\pi_N^{-1}(y))dy.
\end{align*}
By a similar argument, for $y'\in\pi_M^{-1}(y)$, one has $n_T(x,y)=n_T(\pi_M^{-1}(x),y')$. Altogether, we get that
\begin{equation}\label{eq:n_T-M}
h_{\text{top}}(N)\geq\limsup_{T\to\infty}\frac{1}{T}\log\int_{X^{\text{reg}}}n_T(\pi_M^{-1}(x),y'){\text{Vol}}(\pi_N^{-1}(y))dy.\vspace{0.05cm} 
\end{equation}

Note that the horizontal geodesics of length at most $T$ between $\pi_M^{-1}(x)$ and $y'$ correspond to the critical points 
of the energy functional $E:\Omega^T(\pi_M^{-1}(x),y')\to\R$. Since a generic $y'$ is not a focal point of the fiber $\pi_M^{-1}(x)$, the energy functional is a Morse function and, by the Morse inequality, 
$$n_T(\pi_M^{-1}(x),y')\geq\sum_j b_j\left(\Omega^T(\pi_M^{-1}(x),y')\right).$$
But for any fiber $L$ of $\pi_M$, the spaces $\Omega(\pi_M^{-1}(x),y')$ and $\Omega(L,y')$ are homotopy equivalent: any choice of a path
$\gamma$ in $X^{\text{reg}}$ between $x$ and $\pi(L)$ induces a map $\Omega(\pi_M^{-1}(x),y')\to\Omega(L,y')$ sending a curve $c$ to $c\star \bar{\gamma}_c$, where $\bar{\gamma}_{c}$ is the unique horizontal lift of $\gamma$ starting at $c(1)$. It is easy to check that this map is a homotopy equivalence. Hence,
\begin{equation}\label{eq:Morse}
n_T(\pi_M^{-1}(x),y')\geq\sum_j b_j\left(\Omega^T(L,y')\right),
\end{equation}
where $L$ is a fixed fiber of $\pi_M$. Putting together Equations \eqref{eq:n_T-M} and \eqref{eq:Morse}, we get\vspace{0.05cm} 
\begin{align*}
h_{\text{top}}(N) & \geq\limsup_{T\to\infty}\frac{1}{T}\log\int_{X^{\text{reg}}}\left(\sum_j b_j\left(\Omega^T(L,y')\right)\right){\text{Vol}}(\pi_N^{-1}(y))dy\\
 & \geq \limsup_{i\to\infty}\frac{1}{Ci}\log\int_{X^{\text{reg}}}\left(\sum_j b_j\left(\Omega^{Ci}(L,y')\right)\right){\text{Vol}}(\pi_N^{-1}(y))dy,\vspace{0.05cm} 
\end{align*}
\todo{small changes made here (Point (3))}
where $C$ is the constant discussed in Proposition \ref{P:Gromov} satisfying $H_j(\Omega(L,y'))= \iota_*H_j(\Omega^{Ck}(L,y'))$ for every $k$ and every $j\leq k$. We thus have 
$b_j\left(\Omega^{Ci}(L,y'))\right)\geq b_j\left(\Omega(L,y'))\right)$ for all $j\leq i$, and hence\\
$$\sum_jb_j\left(\Omega^{Ci}(L,y')\right)\geq\sum_{j\leq i}b_j\left(\Omega(L,y')\right).\vspace{0.1cm}$$
Therefore, taking $i_0$ to be large enough such that $B_{Ci_0}(x)\supseteq X^{\text{reg}}$, it follows that
\begin{align*}
h_{\text{top}}(N) & \geq\limsup_{i\to\infty}\frac{1}{Ci}\log\int_{B_{Ci}(x)\cap X^{\text{reg}}}\left(\sum_{j\leq i}b_j\left(\Omega(L,y')\right)\right){\text{Vol}}(\pi_N^{-1}(y))dy\\
 & =\limsup_{\substack{i\to\infty\\ i\geq i_0}}\frac{1}{Ci}\log\left(\sum_{j\leq i}b_j\left(\Omega(L,y')\right)\int_{B_{Ci}(x)\cap X^{\text{reg}}}{\text{Vol}}(\pi_N^{-1}(y))dy\right)\\
 & =\limsup_{\substack{i\to\infty\\ i\geq i_0}}\frac{1}{Ci}\log\left(\sum_{j\leq i}b_j\left(\Omega(L,y')\right)\int_{ X^{\text{reg}}}{\text{Vol}}(\pi_N^{-1}(y))dy\right)\\
 & =\limsup_{\substack{i\to\infty\\ i\geq i_0}}\frac{1}{Ci}\log\left(\sum_{j\leq i}b_j\left(\Omega(L,y')\right){\text{Vol}}(N)\right)\\
  & =\limsup_{\substack{i\to\infty\\ i\geq i_0}}\frac{1}{Ci}\left(\log\sum_{j\leq i}b_j\left(\Omega(L,y')\right)+\log{\text{Vol}}(N)\right)\\
 & =\limsup_{\substack{i\to\infty}}\frac{1}{Ci}\log\sum_{j\leq i}b_j\left(\Omega(L,y')\right).
 \end{align*}
 
The above equation, together with the assumption that the topological entropy of $N$ is zero, implies that the sequence 
of the Betti numbers of $\Omega(L,y')$ has sub-exponential growth. 

Now, consider the fibration $\Omega(L,y')\to L\to M$. By assumption, $L$ is nilpotent and satisfies
$\sum_{i\geq 2}\dim\pi_i(L)\tensor\Q<\infty$. Therefore, $M$ is rationally elliptic by Theorem A.1 in \cite{GWY19}.

\end{proof}

We end this section with the proof of the main result of the paper, which is a special case of the following

\begin{proposition}\label{P:generaliz}
\todo{Actual proposition generalized}
Let $M$ be a simply-connected, closed Riemannian $G$-manifold, where $G$ is a compact Lie group. Assume that $M/G$ is isometric to the base of a submetry $\pi_N:N\to X$, where $N$ is a compact Riemannian manifold with zero topological entropy. Then M is rationally elliptic.
\end{proposition}

\begin{proof}
\todo{Proof simplified and kept with same generality as before (Point (4))}
Given a $G$-manifold $M$, we apply the Reduction Lemma 4.5 in \cite{GWY19} to obtain a new simply connected $\hat{G}$-manifold $\hat{M}:= M\times_G\hat{G}$ (where $\hat{G}$ can be taken e.g. $SU(n)$, with $n$ large enough that there exists an embedding $G\to SU(n)$) such that $M/G$ is isometric to $\hat{M}/\hat{G}$, $M$ is rationally elliptic if and only if $\hat{M}$ is, and furthermore the principal orbits of $\hat{G}$ are connected with finite fundamental group (and rationally elliptic universal cover, being a compact simply connected homogeneous space). Thus, Theorem \ref{main-thm: foliations} applies to $\hat{M}$, hence $\hat{M}$ is rationally elliptic, hence so is $M$.
\end{proof}

\section{Non-negatively curved, cohomogeneity three manifolds}\label{S:infinitesimally polar}

In this section, we apply Theorem \ref{main-cor:G-manifolds} to prove rational ellipticity of simply connected, 
non-negatively curved Riemannian manifolds which admit cohomogeneity three, infinitesimally polar actions. 
Before proceeding, we collect some lemmas required for the proof.

\begin{lemma}\label{lem:slice to tubular}
Suppose a Lie group $G$ acts isometrically on a closed Riemannian manifold $(M,\hat g)$. 
Fix $p \in M$, and let $L_p$ denote the orbit through $p$. Assume $N$ is an open subset of $(T_pL_p)^{\perp_{\hat{g}}}$ 
that is invariant under the action of $G_p$. Assume, moreover, that the map $f:(G \times N)/G_p \to M$ sending $[h,w]$ 
to $\exp_{h\cdot p}(h_*w)$ (exponential taken with respect to the metric $\hat g$) is a diffeomorphism onto its image, 
and let $B$ denote the image of $f$. Then for any $G_p$-invariant Riemannian metric $\widetilde g$ on $N$, 
there exists a $G$-invariant Riemannian metric $g$ on $B$ such that the exponential map $(N,\widetilde{g})\to (B,g)$ 
induces an isometry between the quotient spaces $N/{G_p}$ and $B/G$.
\end{lemma}

\begin{proof}
Let $F:B\to L_p$ denote the closest-point projection map with respect to $\hat{g}$. By the assumptions of the Lemma, for $p'\in L_p$ 
and $v\in (T_{p'}L_p)^{\perp_{\hat{g}}}$ such that $\exp_{p'}(v)\in B$, we have $F(\exp_{p'}(v))=p'$. 
Define distributions $\V$ and $\K$ of $TB$, as follows:
\begin{itemize}
\item $\K_q=\ker(d_qF)$.
\item $\V_q=\{X^*_q\mid X\in \g,\, X\perp\g_{F(q)}\}$, where $\g$ and $\g_{F(q)}$ denote the Lie algebras of $G$ 
and $G_{F(q)}$, respectively, and where the orthogonality is with respect to some fixed bi-invariant metric on $G$.
\end{itemize}

Notice the following things about $\V$ and $\K$:
\begin{enumerate}
\item {\bf They are distributions:} on the one hand, it is clear that $\dim\K_q=\operatorname{codim}(L_p)$ is independent of $q$ and that $\K$ is a distribution. As for $\V$, notice that, since the isotropy group $G_q$ is contained in $G_{F(q)}$ for every $q\in B$, one has $\dim \V_q=\dim\V_{F(q)}=\dim G-\dim G_{F(q)}=\dim L_p$ for every $q\in B$. Hence the dimension is constant, and it is easy to see that $\V$ is then a distribution as well.
\item {\bf They are $G$-invariant:} On the one hand, since $F$ is $G$-equivariant, we have $F(h\cdot q)=h\cdot F(q)$ for any $h\in G$, and thus $(d_{h\cdot q}F)\circ h_*=h_*\circ d_qF$, which implies that $h_*(\K_q)\subseteq\K_{h\cdot q}$
and thus $\K$ is $G$-invariant. Furthermore, given any $h\in G$ and letting $h:M\to M$ denote the corresponding smooth map, we have:
\begin{align*}
\V_{h\cdot q}=&\{X^*_{h\cdot q}\mid X\in\g,\, X\perp\g_{h\cdot F(q)}\}\\
=&\{X^*_{h\cdot q}\mid X\in\g,\, X\perp Ad(h)\g_{F(q)}\}&& \textrm{(since $\g_{h\cdot F(q)}=Ad(h)\g_{F(q)}$)}\\
=&\{(Ad(h)Y)^*_{h\cdot q}\mid Y\in\g,\, Y\perp \g_{F(q)}\}&& \textrm{(since $Ad(h):\g\to \g$ is an isometry)}\\
=&\{h_*(Y^*_q)\mid Y\in\g,\, Y\perp \g_{F(q)}\}&& \textrm{(since $h_*(Y^*_q)=(Ad(h)Y)^*_{h\cdot q}$)}\\
=&h_*\V_q.
\end{align*}
Therefore, $\V$ is $G$-invariant as well.
\item {\bf They are complementary}:
\begin{itemize}
\item As observed above, $\dim \V+\dim \K=\dim B$,
\item for all $q\in B$, $\K=\ker d_qF$ while $d_qF|_{\V_q}:\V_q\to \V_{F(q)}=T_{F(q)}L_p$ is injective, therefore $\K\cap \V=0$.
\end{itemize}
\end{enumerate}

By the discussion above, every vector in $TB$ can be written uniquely as $x+v$ where $x\in \K$ and $v\in \V$.
Now, define the metric $g$ on $B$ as follows: given $q=f([h,w])=\exp_{h\cdot p}(h_*w)$, let
$$g_q(x_1+v_1,x_2+v_2)=\hat{g}_q(v_1,v_2)+\widetilde{g}_{\exp_p^{-1}(h^{-1}\cdot q)}((\exp_p^{-1})_*\circ(h^{-1})_*x_1,(\exp_p^{-1})_*\circ(h^{-1})_*x_2).$$
This metric has the following properties:
\begin{enumerate}
\item $g$ is well defined and a smooth metric: given a different representation $q=f([hm^{-1}, m_*w])$ for some $m\in G_p$, 
we have:
{\small
\begin{align*}
\hspace{0.1cm}g_q(x_1+v_1,x_2+v_2) & = \hat{g}_q(v_1,v_2)+\widetilde{g}_{\exp_p^{-1}(mh^{-1}\cdot q)}\left((\exp_p^{-1})_*\circ m_*((h^{-1})_*x_1),(\exp_p^{-1})_*\circ m_*((h^{-1})_*x_2)\right)\\
 & = \hat{g}_q(v_1,v_2)+\widetilde{g}_{\exp_p^{-1}(mh^{-1}\cdot q)}\left((m\circ\exp_p^{-1})_*((h^{-1})_*x_1),(m\circ\exp_p^{-1})_*((h^{-1})_*x_2)\right)\\
 & =\hat{g}_q(v_1,v_2)+\widetilde{g}_{\exp_p^{-1}(h^{-1}\cdot q)}((\exp_p^{-1})_*\circ(h^{-1})_*x_1,(\exp_p^{-1})_*\circ(h^{-1})_*x_2),
\end{align*}}
where the second equality follows from the $G_p$-equivariance of the exponential map and the third equality holds because $\widetilde g$ 
is $G_p$-invariant. This proves well-definedness. Since $\V$ and $\K$ and
 complementary, it follows that $g$ is a metric. Finally, since $\V$ and $\K$ are smooth, the projections onto $\V$ and $\K$ are smooth, which implies that both terms in the definition of $g$ are smooth.
\item $g$ is $G$-invariant: given $z_i=x_i+v_i\in T_qB$, $i=1,2$, with $q=f([h,w])$, and given $m\in G$, 
it follows from $G$-invariance of $\V$ and $\K$ that $m_*z_i=m_*x_i+m_*v_i$ is the decomposition of $m_*z_i$
into $\V$ and $\K$-components. Furthermore, since $m\cdot q=f([mh,w])$, the second term in the expression 
for $g_{m\cdot q}(m_*x_1+m_*v_1,m_*x_2+m_*v_2)$ equals
$$\widetilde{g}_{\exp_p^{-1}\left((mh)^{-1}\cdot (m\cdot q)\right)}((\exp_p^{-1})_*\circ(mh)^{-1}_*(m_*x_1),(\exp_p^{-1})_*\circ(mh)^{-1}_*(m_*x_2)).$$
Altogether, we get that
{\small
\begin{align*}
\hspace{0.3cm}g_{m\cdot q}(m_*z_1,m_*z_2) & =g_{m\cdot q}(m_*x_1+m_*v_1,m_*x_2+m_*v_2)\\
 & =\hat{g}_{m\cdot q}(m_*v_1,m_*v_2)+\widetilde{g}_{\exp_p^{-1}(h^{-1}\cdot q)}((\exp_p^{-1})_*\circ(h^{-1})_*x_1,(\exp_p^{-1})_*\circ(h^{-1})_*x_2)\\
 & =\hat{g}_q(v_1,v_2)+\widetilde{g}_{\exp_p^{-1}(h^{-1}\cdot q)}((\exp_p^{-1})_*\circ(h^{-1})_*x_1,(\exp_p^{-1})_*\circ(h^{-1})_*x_2)\\
 & =g_q(x_1+v_1,x_2+v_2)\\
 & =g_q(z_1,z_2).
\end{align*}
}
\end{enumerate}

We are left to prove that the exponential map $(N,\widetilde g) \to (B, g)$ induces an isometry between the orbit spaces 
$N/{G_p}$ and $B/G$. It is enough to check this on the regular part, so let $y_*\in T_{w_*}(N/G_p)^{\textrm{reg}}$, 
and let $y\in T_wN$ denote a $\widetilde{g}$-horizontal lift of $y_*$ via the canonical projection $N\to N/{G_p}$. 
Define $q:=\exp_p w$, and notice $x:=(\exp_p)_*(y)\in \K_{q}$. In particular, $x$ is $g$-perpendicular to $\V$ 
(since $\V$ and $\K$ are $g$-orthogonal by the definition of $g$) and $x$ is $g$-perpendicular to all action fields $X^*$ 
with $X\in \g_{F(q)}$ (because the restriction of $g$ to $\K$ is equal to $\widetilde{g}$ and $y$ is $\widetilde{g}$-horizontal). By the definition of $\V$, it then turns out that $x$ is $g$-perpendicular to all action fields $X^*$ with $X\in \g$, 
or equivalently, $x$ is $g$-horizontal. Let $q_*$ denote the image of $q$ in $B/G$, and let $x_*\in T_{q_*}(B/G)$ 
denote the projection of $x$. We conclude that
$$||x_*\|_{B/G}\stackrel{(1)}{=}\|x\|_g\stackrel{(2)}{=}\|y\|_{\widetilde{g}}=\|y_*\|_{N/{G_p}},$$
where equality $(1)$ follows from the fact that $(B,g)\to B/G$ is a Riemannian submersion on the regular part and $x$ 
is a $g$-horizontal lift of $x_*$, and $(2)$ follows from the definition of $g$. This proves the lemma.
\end{proof}

The following result is briefly discussed in a Remark in \cite{GWY19}. We provide a more elaborate proof here.

\begin{proposition}\label{prop:orbifold metric}
Suppose $(M,\hat{g})$ is a closed Riemannian manifold which admits an isometric action by a Lie group $G$ 
such that the quotient space $M/G$ is an orbifold. Then given an orbifold metric $g_0$ on $M/G$, there exists a $G$-invariant Riemannian 
metric $g$ on $M$ such that the projection $\pi:(M,g)\to (M/G,g_0)$ is a
\todo{changed word (Point (5))}
 submetry.
\end{proposition}

\todo{Proof divided into two, with proofs modified, according to referee's suggestions (Point (8))}
The proof is divided in two parts: first, in Proposition \ref{P:Step1}, we show how to lift metrics locally around $G$-orbits. Next, in Proposition \ref{P:Step2}, we show how to glue local metrics using Lemma 3.4 in \cite{GWY19}, but giving a more global argument for it.

\begin{proposition}\label{P:Step1}
Suppose $(M,\hat{g})$ is a closed Riemannian manifold which admits an isometric action by a Lie group $G$ 
such that the quotient space $M/G$ is isometric to a Riemannian orbifold. Let $\pi:M\to M/G$ be the quotient map, and let $U=B_{\epsilon}(p_*)$ be a small neighbourhood of $p_*\in M/G$. Then for any orbifold metric $g_0$ on $U$, there exists a $G$-invariant Riemannian 
metric $g$ on $\pi^{-1}(U)$ such that $\pi|_{\pi^{-1}(U)}$ is a submetry.
\end{proposition}

\begin{proof}
Choose $p\in \pi^{-1}(p_*)$. By Lytchak-Thorbergsson \cite{LT10}, the action of $G_p$ on $\nu_pL_p$ is polar. In particular, there is a totally geodesic section $\Sigma_p\subseteq \nu_pL_p$ acted on by a finite group $W_p$ so that the inclusion $\Sigma_p\embedded \nu_pL_p$ induces a homeomorphism ${\Sigma_p}/W_p\to{\nu_pL_p}/G_p$ (cf. Section \ref{S:polar}). By Theorem 1.4 in \cite{LT10}, for any $W_p$-invariant neighborhood $\bar{U}\subseteq \Sigma_{p}$ of the origin the composition $\bar{U}\to {\bar{U}}/{W_p}\to U$ is an orbifold chart for $M/G$. In particular, there is a $W_{p}$-invariant metric $\bar{g}$ on $\bar{U}$ so that the map $\varphi:{\bar{U}}/{W_p}\to U$ induced by the map above becomes an isometry.

By Theorem 1.2 in \cite{Men16}, there exists a $G_{p}$-invariant metric $\widetilde{g}$ on $N:=\cup_{h\in G_{p}}h\cdot\bar{U}$ such that the maps $(N,\widetilde{g})/G_p\to (\bar{U},\bar{g})/W_p\to (U,g_0)$ are all isometries. By Lemma \ref{lem:slice to tubular}, there exists a $G$-invariant metric $g$ on $\pi^{-1}(U)$ such that the map $(N,\widetilde{g})/G_p\to (\pi^{-1}(U),g)/G$ is an isometry. In particular, the projection $(\pi^{-1}(U),g)\to (\pi^{-1}(U),g)/G\simeq (U, g_0)$ is a manifold submetry.
\end{proof}

For the next result, recall that given a Riemannian metric $g$ on a manifold $M$, its \emph{cometric} $g^*$ is the Euclidean structure on $T^*M$ obtained declaring an orthonormal basis of $T^*_pM$ as the dual basis of an orthonormal basis of $T_pM$.

\begin{proposition}\label{P:Step2}
Let $X=X_1\cup X_2$ be a metric space, $M=M_1\cup M_2$ be a manifold, and $\pi:M\to X$ a map with smooth fibers. Let $(M_i, g_i)$, $i=1,2$ be Riemannian metrics such that $\pi_i:(M_i,g_i)\to X_i$ are manifold submetries, and let $\{f_1, f_2\}$ a basic partition of unity subordinate to $\{M_1, M_2\}$. Then, letting $g$ be the Riemannian metric on $M$ whose cometric is $g^*=f_1g_1^*+f_2g_2^*$, we have that $\pi:(M,g)\to X$ is a manifold submetry as well.
\end{proposition}
\begin{proof}
\todo{New proof, improved from the version of Grove Yeager Wilking (Point (8))}
Let $X^{pr}$ be the manifold part (or \emph{principal part}) of $X$, with Riemannian metric $g_0$, and let $M^{pr}:=\pi^{-1}(X^{pr})$. Around $p_*\in X^{pr}\cap X_i$ and $p\in \pi^{-1}(p_*)\in M_i$, the map $\pi$ is a Riemannian submersion and, by definition, the adjoint map $(d_p\pi)^*:(T_{p_*}^*X,g_0)\to (T_p^*M_i,g_i^*)$ is an isometric immersion. In fact, the image of $(d_p\pi)^*$ is $\operatorname{Ann}(T_pL_p):=\{\alpha\in T_p^*M\mid \alpha(T_pL_p)=0\}$ \emph{independently of the metric}. In particular, if $p\in M^{pr}\cap M_1\cap M_2$, then
\[
(d_p\pi)^*:(T_{p_*}^*X,g_0)\to (\operatorname{Ann}(T_pL_p),g_i^*)
\]
is an isometry with respect to both cometrics. In particular, if $g^*=f_1(p)g_1^*+f_2(p)g_2^*$, then $(d_p\pi)^*:(T_{p_*}^*X,g_0)\to (\operatorname{Ann}(T_pL_p),g^*)$ is also an isometry, that is, $\pi:(M^{pr},g)\to (X^{pr},g_0)$ is a Riemannian submersion. Since all fibers of $\pi$ are already assumed to be smooth, and general fibers of $\pi$ are Hausdorff limits of principal fibers, we obtain that $\pi:(M,g)\to X$ is a manifold submetry.
\end{proof}

With Proposition \ref{prop:orbifold metric} in hand, we proceed to the proof of Theorem \ref{main-thm:infinitesimally polar}.

\begin{proof}[Proof of Theorem \ref{main-thm:infinitesimally polar}]
\todo{Proof made precise, making sure all metrics involved are invariant (Point (9))}
Since the $G$-action on $M$ is infinitesimally polar, the quotient space $M/G$ is a non-negatively curved orbifold by result of Lytchak-Thorbergsson \cite{LT10}. By the classification of compact, non-negatively curved $3$-orbifolds in \cite{KL14}, Proposition 5.7, $M/G$ is diffeomorphic to to one of:
\begin{itemize}
\item  $N'/\Gamma$, with $N'=S^3$, or $T^3$ with their constant sectional curvature metrics,
\item  $N'/\Gamma$ where $N'=S^1\times S^2$ is endowed with the product of round metrics and $\Gamma=1\times \Gamma'$, $\Gamma'\subset \operatorname{Isom}(S^2)$ finite.
\item the metric product $S^1\times (S^2_{p,q}/\Gamma)$ where $S^2_{p,q}$ is a simply-connected bad 2-orbifold equipped with its unique Ricci soliton metric, and $\Gamma\subset \operatorname{Isom}(S^2_{p,q})$ is finite.
\end{itemize}

In the first two cases, it is known that the metrics for $N'$ above have $h_{top}(N')=0$. In these cases Proposition \ref{prop:orbifold metric} gives a $G$-invariant metric on $M$ such that $M/G$ is isometric to $N'/\Gamma$, therefore $M$ is rationally elliptic by Theorem \ref{main-cor:G-manifolds}.

In the last case of $N'=S^1\times S^2_{p,q}$, this orbifold is diffeomorphic to $S^1\times (S^3/S^1)$ where $S^1\subset \C$ acts on $S^3\subset \C^2$ via the weighted Hopf action $w\cdot (z_1, z_2)=(w^pz_1,w^qz_2)$. As such, there is a rotationally symmetric metric $(S^2_{p,q},g_{quot})$ such that $(S^3, g_{round})\to (S^2_{p,q},g_{quot})$ is a submetry. Furthermore, the Ricci soliton metric $(S^2_{p,q},g_{sol})$ is rotationally symmetric as well,  and with the same two fixed points. In particular, we can identify the isometry groups of $(S^2_{p,q},g_{sol})$ and $(S^2_{p,q},g_{quot})$ (thus $\Gamma\subset \operatorname{Isom}(S^2_{p,q}, g_{quot})$). Letting $N=S^1\times S^3$ with the product of round metrics, we can consider $X:= S^1\times (S^2_{p,q}/\Gamma)$ with the metric induced by $g_{quot}$, as the base of the manifold submetry $N\to N'\to X$ with $h_{top}(N)=0$. Proposition \ref{prop:orbifold metric} gives a $G$-invariant metric on $M$ such that $M/G$ is isometric to $X$, and from Proposition \ref{P:generaliz} we have that $M$ is rationally elliptic.


\end{proof}

\bibliographystyle{plain}		

\end{document}